\documentclass[11pt, a4paper]{article}

\usepackage{amsmath,amssymb}
\usepackage[utf8]{inputenc}
\usepackage[colorlinks=true, citecolor=blue]{hyperref}
\usepackage{amsthm}
\usepackage{enumerate}
\usepackage{fullpage}
\usepackage{color}
\usepackage{graphicx}
\usepackage{mathtools}
\usepackage{subcaption}
\usepackage{tikz}
\usetikzlibrary{decorations.markings}
\numberwithin{equation}{section}
\usepackage{float}
\usepackage[capitalize,nameinlink,noabbrev]{cleveref}
\usepackage{array}

\usepackage[affil-it]{authblk}

\newtheorem{theorem}{Theorem}
\newtheorem{lemma}[theorem]{Lemma}
\newtheorem{corollary}[theorem]{Corollary}

\usepackage{array}
\usepackage{verbatim}

\newcommand{\olx}{\overline{x}}
\newcommand{\oly}{\overline{y}}
\newcommand{\olxy}{\overline{xy}}

\DeclareMathOperator{\bigO}{O}

\title{Quarter-plane lattice paths with interacting boundaries: the Kreweras and reverse Kreweras models}

\author{Nicholas R. Beaton\thanks{\href{mailto:nrbeaton@unimelb.edu.au}{nrbeaton@unimelb.edu.au}}}
\author{Aleksander L. Owczarek\thanks{\href{mailto:owczarek@unimelb.edu.au}{owczarek@unimelb.edu.au}}}
\author{Ruijie Xu\thanks{\href{mailto:ruijiex1@student.unimelb.edu.au}{ruijiex1@student.unimelb.edu.au}}}

\affil{\small School of Mathematics and Statistics\\The University of Melbourne, VIC 3010, Australia}

\RequirePackage[maxbibnames=99,sorting=nyt,giveninits,maxnames=10,backend=bibtex]{biblatex}
\renewbibmacro{in:}{}
\renewcommand*{\bibfont}{\small}
\DeclareFieldFormat[misc]{title}{\mkbibquote{#1}}
\DeclareFieldFormat[article]{volume}{\mkbibbold{#1}}
\DeclareFieldFormat{url}{\href{#1}{Link}}
\DeclareFieldFormat{doi}{\href{http://dx.doi.org/#1}{Link}}
\DeclareFieldFormat{eprint:arxiv}{\href{https://arxiv.org/abs/#1}{arXiv}}
\AtEveryBibitem{\clearfield{issn}\clearfield{isbn}}

\begin{document}
\maketitle
\abstract{
Lattice paths in the quarter plane have led to a large and varied set of results in recent years. One major project has been the classification of step sets according to the properties of the corresponding generating functions, and this has involved a variety of techniques, some highly intricate and specialised. The famous Kreweras and reverse Kreweras walk models are two particularly interesting models, as they are among the only four cases which have algebraic generating functions.

Here we investigate how the properties of the Kreweras and reverse Kreweras models change when boundary interactions are introduced. That is, we associate three real-valued weights $a,b,c$ with visits by the walks to the $x$-axis, the $y$-axis and the origin $(0,0)$ respectively. These models were partially solved in a recent paper by Beaton, Owczarek and Rechnitzer (2019). We apply the algebraic kernel method to completely solve these two models. We find that reverse Kreweras walks have an algebraic generating function for all $a,b,c$, regardless of whether the walks are restricted to end at the origin or on one of the axes, or may end anywhere at all. For Kreweras walks, the generating function for walks returning to the origin is algebraic, but the other cases are only D-finite. To our knowledge this is the first example of a quarter-plane model with this property.

\

\noindent \textbf{Keywords:} random walks; enumeration; generating functions; kernel method; algebraic functions; D-finite functions
}

\section{Introduction}\label{sec:intro}

Lattice path models with boundary conditions have been studied widely. As a combinatorial problem, they are closely related to probability theory, algebra, complex analysis and statistical physics~\cite{bousquet2005walks,bousquet2010walks,fayolle1999random,kurkova2012functions,mishna2009classifying}. The typical goal is to study the properties of random walks with steps in a fixed set $S$. These properties include the number of paths of a certain length, the generating function, the asymptotic behaviour and bijections with other combinatorial objects.

The generating function of simple walks in the bulk can be written down directly with simple calculations. Once the boundary conditions are introduced, the problem becomes more complicated and some interesting results appear. In~\cite{banderier2002basic}, Banderier and Flajolet prove that for walks in a half plane (walks with one boundary constraint), the generating function is always algebraic. For quarter-plane walks with small steps, Bousquet-M{\'e}lou and Mishna~\cite{bousquet2010walks} associated each step set $S$ with a group, and proved that among all $79$ non-isomorphic (and non-trivial) quarter-plane models, exactly $23$ have a finite group and the remaining $56$ have an infinite group. A walk model with a finite group can be solved by the kernel method~\cite{prodinger2004kernel} and all of them have D-finite generating functions. Four specific models: Kreweras, reverse Kreweras, double Kreweras, and Gessel, have algebraic generating functions. Models whose associated group is infinite are far more difficult to solve (find an explicit expression for the generating function), but many properties can still be discussed~\cite{dreyfus_nature_2018,mishna2009two}.
Recently, walks avoiding a quarter-plane have also been studied~\cite{bousquet2016square,raschel_walks_2019}.

The connection with statistical physics can be seen in a recent publication~\cite{tabbara2013exact}. The authors  studied a two-dimensional model of  interacting directed polymers above an impenetrable surface. Two walks starting at the origin can walk north-east or south-east with step length $1$. Weights (Boltzmann weights) are assigned when the two walks touch or when any either touches the surface.

This interacting directed walk model has a bijection to a certain quarter-plane path problem (see \cref{fig0}). Let the generating variable of half distance between two walks be $s$ and the generating variable of half distance between the lower walk and the surface be $r$. If we have the following bijection:
\[\genfrac(){0pt}{0}{\nearrow}{\nearrow} \mapsto (\rightarrow) \qquad \genfrac(){0pt}{0}{\nearrow}{\searrow} \mapsto (\nwarrow) \qquad \genfrac(){0pt}{0}{\searrow}{\searrow} \mapsto (\leftarrow) \qquad \genfrac(){0pt}{0}{\searrow}{\nearrow} \mapsto (\searrow)\]
 Then, the directed walk with generating variable $(s,r)$ corresponds to a quarter-plane walk whose allowed steps are east, west, northwest and southeast. This quarter-plane walk model has interactions when the walk touches the $x$ axis and the $y$ axis. The weights are the same as in the corresponding directed walk model.

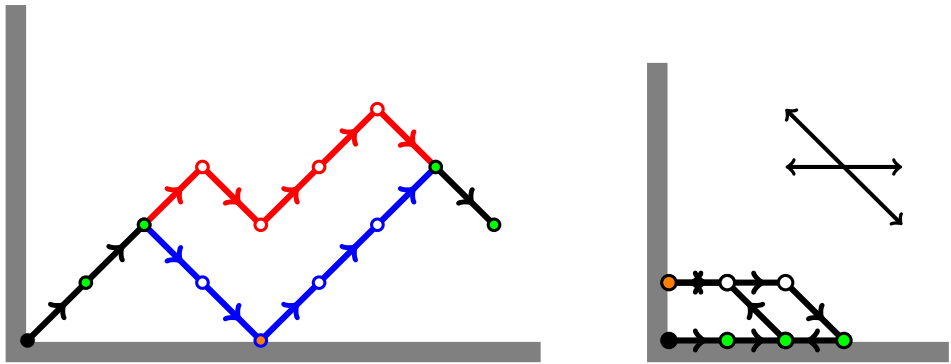
\begin{figure}
\centering
\resizebox{0.8\textwidth}{!}{
\begin{tikzpicture}
\tikzset{12/.style={circle, line width=1.5pt, draw=black, fill=green, inner sep=2pt}}
\tikzset{1/.style={circle, line width=1.5pt, draw=blue, fill=white, inner sep=2pt}}
\tikzset{2/.style={circle, line width=1.5pt, draw=red, fill=white, inner sep=2pt}}
\tikzset{1x/.style={circle, line width=1.5pt, draw=blue, fill=orange, inner sep=2pt}}
\tikzset{12x/.style={circle, line width=1.5pt, draw=black, fill=black, inner sep=2pt}}
\draw [gray, line width=10pt] (-0.2,5.8) -- (-0.2,-0.2) -- (8.8,-0.2);
\begin{scope}[line width=3pt, decoration={markings,mark=at position 0.65 with {\arrow{>}}}]
\draw [postaction=decorate] (0,0) node [12x] {} -- (1,1);
\draw  [postaction=decorate] (1,1) node [12] {} -- (2,2);
\draw [blue, postaction=decorate] (2,2) node [12] {} -- (3,1);
\draw [blue, postaction=decorate] (3,1) node [1] {} -- (4,0);
\draw [blue, postaction=decorate] (4,0) node [1x] {} -- (5,1);
\draw [blue, postaction=decorate] (5,1) node [1] {} -- (6,2);
\draw [blue,postaction=decorate] (6,2) node [1] {} -- (7,3);
\draw [red, postaction=decorate] (2,2)  -- (3,3);
\draw [red, postaction=decorate] (3,3) node [2] {} -- (4,2);
\draw [red, postaction=decorate] (4,2) node [2] {} -- (5,3);
\draw [red, postaction=decorate] (5,3) node [2] {} -- (6,4);
\draw [red, postaction=decorate] (6,4) node [2] {} -- (7,3);
\draw [ postaction=decorate] (7,3) node [12] {} -- (8,2);
\draw  (2,2) node [12] {};
\draw  (8,2) node [12] {};
\end{scope}
\begin{scope}[xshift=11cm]
\tikzset{ac/.style={circle, line width=1.5pt, draw=black, fill=green, inner sep=2.5pt}}
\tikzset{bc/.style={circle, line width=1.5pt, draw=black, fill=orange, inner sep=2.5pt}}
\tikzset{cc/.style={circle, line width=1.5pt, draw=black, fill=black, inner sep=2.5pt}}
\tikzset{vert/.style={circle, line width=1.5pt, draw=black, fill=white, inner sep=2.5pt}}
\draw [gray, line width=10pt] (-0.2,4.8) -- (-0.2,-0.2) -- (4.8,-0.2);
\begin{scope}[line width=3pt, decoration={markings,mark=at position 0.65 with {\arrow{>}}}]
\draw [postaction=decorate] (0,0) node [cc] {} -- (1,0);
\draw [postaction=decorate] (1,0) node [ac] {} -- (2,0);
\draw [postaction=decorate] (2,0) node [ac] {} -- (1,1);
\draw [postaction=decorate] (1,1) node [vert] {} -- (0,1);
\draw [postaction=decorate] (0,1) node [bc] {} -- (1,1);
\draw [postaction=decorate] (1,1) node [vert] {} -- (2,1);
\draw [postaction=decorate] (2,1) node [vert] {} -- (3,0);
\draw [postaction=decorate] (3,0) node [ac] {} -- (2,0);
\draw  (2,0) node [ac] {};
\end{scope}
\draw [line width=2pt, <->] (2,3) -- (3,3) -- (4,2);
\draw [line width=2pt, <->] (2,4) -- (3,3)--(4,3);
\end{scope}
\end{tikzpicture}
}
\caption{Left: A pair of interacting directed paths with weights associated with shared vertices and visits to the bottom boundary. Right: The corresponding quarter-plane lattice path with weights associated with visits to the two boundaries.}\label{fig0}
\end{figure}

We may expand this idea. As there is a rich body of literature on non-interacting quarter-plane walks, what happens if we introduce the interaction into other quarter-plane walk problems? In a recent work~\cite{beaton_exact_2019}, Beaton, Owczarek and Rechnitzer introduced interacting boundary weights to quarter-plane walk models and studied the $23$ models which are associated with finite groups. In particular, they assigned weights $a$ and $b$ to the two axes and $ab$ to the origin. For some of the models, the generating function stays D-finite for all $a,b$. For some models, the generating function may be D-finite, may be D-algebraic or may be unsolvable depending on the values of $a$ and $b$.

In this paper, we focus on two specific cases: the reverse Kreweras and Kreweras models. It has been proved that the generating functions of these two models are algebraic without interactions, or with equal interactions  on two boundaries (that is, $a=b$)~\cite{beaton_exact_2019,bousquet2010walks}. For arbitrary $a,b$, it is still unclear, and it is this general case that we address here. (In fact we generalise further by associating weight $c$ with visits to the origin, instead of $ab$.) For these two models we obtain solutions by delicately combining two commonly-used tools -- the \emph{full-orbit sum} and the \emph{half-orbit sum} -- only one of which (the half-orbit sum) could be used for the unweighted and $a=b$ cases. The two other models which are algebraic in the absence of boundary weights -- double Kreweras and Gessel -- are technically more challenging and still remain unsolved.

The layout of this paper is as follows. In \cref{sec:themodel}, we will define some notation used in this paper and recall some general definitions. We will present the final result in \cref{sec:mainresult}. In \cref{sec:revkrew,sec:krew} we will walk through the whole process of the algebraic kernel method and solve the two problems. We will prove that the generating function of reverse Kreweras walks is still algebraic for all boundary weights, while the generating function of Kreweras walk is shown to be D-finite (and probably not algebraic).
In \cref{sec:discussion} we will discuss some remaining open problems.

Many calculations in this paper involve complicated rational, algebraic or D-finite expressions, some of which would take multiple pages to be written down explicitly. For this reason we have produced {\sc Mathematica} notebooks which go through all calculations and verify all equations. These are available at the first author's website.\footnote{\url{http://www.nicholasbeaton.com/papers}.}

\section{The model}\label{sec:themodel}

\subsection{Definitions and notation}

We first define some notation used in this paper. We write $[x^i]f(x)$ for the coefficient of $[x^i]$ in the Laurent series expansion of $f(x)$. Respectively, $[x^>]f(x)$, $[x^<]f(x)$ and $[x^\geq]f(x)$ are those terms with positive, negative and non-negative powers of $x$. We use the notation $\olx=x^{-1}$ and $\oly=y^{-1}$.

For a ring $\mathbb{K}$, we denote
\begin{enumerate}
\item $\mathbb{K}(t)$ as the set of polynomials in $t$ with coefficients in $\mathbb{K}$;
\item $\mathbb{K}((t))$ as the set of polynomials in $t$ and $1/t$ with coefficients in $\mathbb{K}$;
\item $\mathbb{K}[t]$ as the set of formal power series in $t$ with coefficients in $\mathbb{K}$;
\item $\mathbb{K}[[t]]$ as the set of Laurent series in $t$ with coefficients in $\mathbb{K}$;
\item $\mathbb{K}^{r}(t)$ as the set of rational functions in $t$ with coefficients in $\mathbb{K}$.
\end{enumerate}
The definition may extend to multiple variables. For example $\mathbb{R}^{r}(x)[[t]]$ refers to the set of  Laurent series in $t$ with coefficients in the ring of rational polynomials of $x$ with real coefficients.

Next, we define functions for counting lattice paths with site interactions on the boundaries.

We denote $q_{n,k,l,h,v,u}$ as the number of walks of length $n$ that start at $(0,0)$ and end at $(k,l)$ which visit the horizontal boundary (except the origin) $h$ times, the vertical boundary (except the origin) $v$ times and the origin $u$ times. The associated generating function is
\begin{equation}
Q(t;x,y;a,b,c)\equiv Q(x,y)=\sum_{n}t^n\sum_{k,l,h,v,u}q_{n,k,l,h,v,u}x^k y^l a^h b^v c^u\equiv \sum_{n}t^n Q_n(x,y).\label{function}
\end{equation}
We also define line-boundary terms:
\begin{equation}
[y^i]Q(t;x,y;a,b,c)\equiv Q_{-,i}(x)=\sum_{n}t^n\sum_{k,h,v,u}q_{n,k,i,h,v,u}x^k a^h b^v c^u.
\end{equation}
This is the generating function of walks ending on the line $y=i$. $Q_{i,-}(y)$ is defined similarly.

Furthermore, we define the generating functions of walks ending on a diagonal line:
\begin{equation}
 Q^d_j(x)=\sum_{n}t^n\sum_{i,h,v,u}q_{n,i,i+j,h,v,u}x^i a^hb^v c^u.
\end{equation}
Note that the variable $x$ in $Q^d_j(x)$ marks the $x$-coordinate of the endpoint of walks.

Finally we define point boundary terms:
\begin{equation}
[x^iy^j]Q(t;x,y;a,b,c)\equiv Q_{i,j}=\sum_{n}t^n\sum_{h,v,u}q_{n,i,j,h,v,u}a^hb^v c^u.
\end{equation}
This is the generating function of walks ending at point $(i,j)$.

\subsection{The general functional equation}

Consider a walk starting from the origin with allowed steps $\mathcal{S} \subseteq \{-1,0,1\}^2$. This set of steps is usually denoted $\{\mathrm{N,S,W,E,NE,NW,SE,SW}\}$.

The \emph{step generator} $S$ is
\begin{equation}
S(x,y) = \sum_{(i,j)\in\mathcal{S}}x^i y^j.
\end{equation}
This can be written as
\begin{equation}
S(x,y)=A_{-1}(x)\oly+A_0(x)+A_1(x)y=B_{-1}(y)\olx+B_0(y)+B_{1}(y)x\label{S}
\end{equation}
where for example $A_{-1}(x)$ is $[y^{-1}]S(x,y)$, which refers to the steps going southwards, including $\{\mathrm{S,SE,SW}\}$ steps.

Since the quarter-plane walk is restricted in the first quadrant, it is allowed to touch the axes but not cross them. We denote  $A(x,y)=A_{-1}(x)\oly$ as the illegal steps crossing the $x$-axis, $B(x,y)=B_{-1}(y)\olx$ as the illegal steps crossing the $y$-axis and $G(x,y)=[x^{-1}y^{-1}]S(x,y)$ as the illegal steps crossing the origin diagonally.

By the geometric properties of quarter-plane walks and the boundary conditions, we can derive a functional equation satisfied by the generating function. 
\begin{theorem}
For a lattice walk restricted to the quarter-plane, starting from the origin, with weight $a$  associated with vertices on the $x$-axis (except the origin), weight $b$ associated with vertices on $y$-axis (except the origin) and weight $c$ associated with the origin, the generating function $Q(x,y)$ satisfies the following functional equation:
\begin{multline}
K(x,y)Q(x,y)=\frac1c + \frac1a(a-1 - taA(x,y))Q(x,0) + \frac1b(b - 1 - tbB(x,y))Q(0,y) \\
+\left(\frac{1}{abc}(ac+bc-ab-abc)+tG(x,y)\right)Q(0,0)\label{functional}
\end{multline}
where $K(x,y)=1-tS(x,y)$. It is called the \emph{kernel} of the generating function.
\end{theorem}
\begin{proof}
This equation can be constructed the same way as~\cite[Theorem~6]{beaton_exact_2019}, except the weight at the origin is $c$ here instead of $ab$. 
\end{proof}
If we let $a=b=c=1$, we get the functional equation for quarter-plane walks without interactions. 

\section{Main results}\label{sec:mainresult}

For reverse Kreweras walks, the allowed steps are $\{\mathrm{SW,N,E}\}$. So we have 
\begin{equation}
S(x,y) = x + y + \olxy,
\end{equation} 
and hence
\begin{equation}
A(x,y) = B(x,y) = G(x,y) = \olxy.
\end{equation}
The functional equation~\eqref{functional} becomes
\begin{multline}
\left(1-t(x+y+\olxy)\right)Q(x,y)=\frac1c + \frac1a(a-1 - ta\olxy)Q(x,0) + \frac1b(b - 1 - tb\olxy)Q(0,y) \\
+\left(\frac{1}{abc}(ac+bc-ab-abc)+t\olxy\right)Q(0,0).\label{1x1x}
\end{multline}

Our aim is to solve this equation. From previous work~\cite{bousquet2005walks,bousquet2010walks,mishna2009classifying}, we know that when $a=b=c=1$, the generating function of reverse Kreweras walks is algebraic. We can use both algebraic kernel method and obstinate kernel method to prove this. Here, we extend the result to a more general case: 
\begin{theorem}\label{1111}
For arbitrary $a,b,c$, the generation function $Q(x,y)$ of reverse Kreweras walks is algebraic.
\end{theorem}
We will solve reverse Kreweras walk generating function using the algebraic kernel method in the next section. The final solution is an algebraic expression.

\begin{corollary}\label{cor:krew00}
For Kreweras walks, the generating function $Q_{0,0} \equiv Q(0,0) $ is algebraic for all $(a,b,c)$.
\end{corollary}
\begin{proof}
If we just reverse the direction of each step of a Kreweras walk starting from the origin and ending at the origin, we got a reverse Kreweras walk. So $Q(0,0)$ is the same for these two models. Thus $Q(0,0)$ for Kreweras walks is algebraic. See \cref{fig1}.
\end{proof}
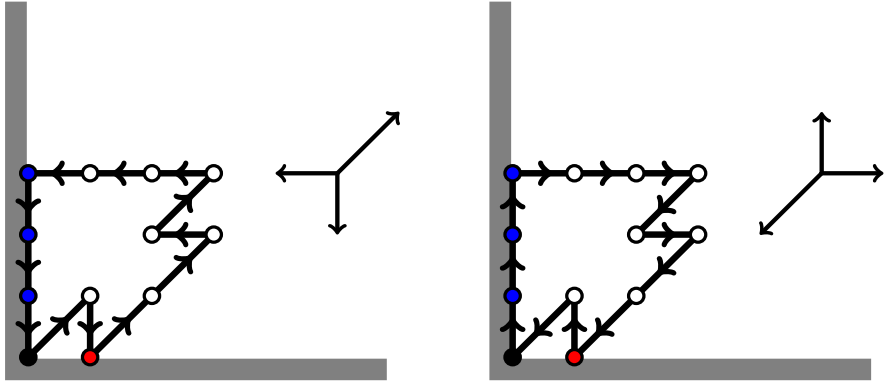
\begin{figure}
\centering
\resizebox{0.8\textwidth}{!}{
\begin{minipage}[t]{0.48\textwidth}
\centering
\begin{tikzpicture}
\tikzset{ac/.style={circle, line width=1.5pt, draw=black, fill=red, inner sep=2.5pt}}
\tikzset{bc/.style={circle, line width=1.5pt, draw=black, fill=blue, inner sep=2.5pt}}
\tikzset{cc/.style={circle, line width=1.5pt, draw=black, fill=black, inner sep=2.5pt}}
\tikzset{vert/.style={circle, line width=1.5pt, draw=black, fill=white, inner sep=2.5pt}}
\draw [gray, line width=10pt] (-0.2,5.8) -- (-0.2,-0.2) -- (5.8,-0.2);
\begin{scope}[line width=3pt, decoration={markings,mark=at position 0.65 with {\arrow{>}}}]
\draw [postaction=decorate] (0,0) node [cc] {} -- (1,1);
\draw [postaction=decorate] (1,1) node [vert] {} -- (1,0);
\draw [postaction=decorate] (1,0) node [ac] {} -- (2,1);
\draw [postaction=decorate] (2,1) node [vert] {} -- (3,2);
\draw [postaction=decorate] (3,2) node [vert] {} -- (2,2);
\draw [postaction=decorate] (2,2) node [vert] {} -- (3,3);
\draw [postaction=decorate] (3,3) node [vert] {} -- (2,3);
\draw [postaction=decorate] (2,3) node [vert] {} -- (1,3);
\draw [postaction=decorate] (1,3) node [vert] {} -- (0,3);
\draw [postaction=decorate] (0,3) node [bc] {} -- (0,2);
\draw [postaction=decorate] (0,2) node [bc] {} -- (0,1);
\draw [postaction=decorate] (0,1) node [bc] {} -- (0,0);
\end{scope}
\draw [line width=2pt, <->] (5,2) -- (5,3) -- (6,4);
\draw [line width=2pt, ->] (5,3) -- (4,3);
\end{tikzpicture}
\label{fig:paths}
\end{minipage}
\begin{minipage}[t]{0.48\textwidth}
\centering
\begin{tikzpicture}
\tikzset{ac/.style={circle, line width=1.5pt, draw=black, fill=red, inner sep=2.5pt}}
\tikzset{bc/.style={circle, line width=1.5pt, draw=black, fill=blue, inner sep=2.5pt}}
\tikzset{cc/.style={circle, line width=1.5pt, draw=black, fill=black, inner sep=2.5pt}}
\tikzset{vert/.style={circle, line width=1.5pt, draw=black, fill=white, inner sep=2.5pt}}
\draw [gray, line width=10pt] (-0.2,5.8) -- (-0.2,-0.2) -- (5.8,-0.2);
\begin{scope}[line width=3pt, decoration={markings,mark=at position 0.65 with {\arrow{>}}}]
\draw [postaction=decorate] (0,0) node [cc] {} -- (0,1);
\draw [postaction=decorate] (0,1) node [bc] {} -- (0,2);
\draw [postaction=decorate] (0,2) node [bc] {} -- (0,3);
\draw [postaction=decorate] (0,3) node [bc] {} -- (1,3);
\draw [postaction=decorate] (1,3) node [vert] {} -- (2,3);
\draw [postaction=decorate] (2,3) node [vert] {} -- (3,3);
\draw [postaction=decorate] (3,3) node [vert] {} -- (2,2);
\draw [postaction=decorate] (2,2) node [vert] {} -- (3,2);
\draw [postaction=decorate] (3,2) node [vert] {} -- (2,1);
\draw [postaction=decorate] (2,1) node [vert] {} -- (1,0);
\draw [postaction=decorate] (1,0) node [ac] {} -- (1,1);
\draw [postaction=decorate] (1,1) node [vert] {} -- (0,0);
\end{scope}
\draw [line width=2pt, <->] (4,2) -- (5,3) -- (5,4);
\draw [line width=2pt, <-] (6,3) -- (5,3);
\end{tikzpicture}
\end{minipage}
}
\caption{Examples of reverse Kreweras and Kreweras walks ending at the origin.}\label{fig1}
\end{figure}
This property does not hold for general $Q_{i,j}$. However, it will play an important role when solving Kreweras walks.

For Kreweras walks, the allowed steps are $\{\mathrm{NE,S,W}\}$. So we have
\begin{equation}
S(x,y) = xy + \olx + \oly
\end{equation}
and hence
\begin{equation}
A(x,y) = \oly, \qquad B(x,y) = \olx, \qquad G(x,y) = 0.
\end{equation}
The functional equation reads
\begin{multline}
\left(1-t(xy+\olx+\oly)\right)Q(x,y)=\frac1c + \frac1a(a-1 - ta\oly)Q(x,0) + \frac1b(b - 1 - tb\olx)Q(0,y) \\
+\frac{1}{abc}(ac+bc-ab-abc)Q(0,0).\label{eqn:krew_func_eqn}
\end{multline}

Our main result regarding Kreweras walks is the following.
\begin{theorem}\label{thm:krew_dfinite}
For arbitrary $a,b,c$, the generation function $Q(x,y)$ of Kreweras walks is D-finite.
\end{theorem}
We will prove this theorem in \cref{sec:krew}. The general idea is the same as for reverse Kreweras walks. We will make use of \cref{cor:krew00} in this proof.

\section{Reverse Kreweras walks}\label{sec:revkrew}

In our situation, the basic idea of the algebraic kernel method will be to find a linear equation in $Q(x,0)$, $Q^d_0(\olx)$ and $Q_{i,j}$ where the positive and negative powers of $x$ can be separated. One can review the walk-through of algebraic kernel method by solving the Kreweras or reverse Kreweras walks with $a=b=c=1$~\cite{bousquet2005walks,bousquet2010walks,mishna2009classifying}. For arbitrary $a,b,c$, we still follow the same ideas but the process is more involved.

The whole process of solving reverse Kreweras walks consists of three steps. 
\begin{itemize}
\item First, we recall the symmetry group of the kernel and use it to take the full-orbit sum. We extract the $[x^{>}y^0]$ and $[x^{<}y^0]$ parts of the full-orbit sum to obtain two equations (\cref{lem:revkrew_fos_pos_and_neg}).
\item Second, we take a half-orbit sum and again take the $[y^0]$ terms. We use the equations obtained from the full-orbit sum to eliminate certain boundary terms, and then once again take the positive and negative parts with respect to $x$. This yields two equations which have kernel-like form (\cref{lem:revkrew_two_kernel_equations}).
\item  Finally, we cancel the kernels of these two equations (see \cref{lem:revkrew_twokernels_roots}) and find a set of linear equations. These equations will provide us the final solution.
\end{itemize}

This process should be contrasted with the solution to reverse Kreweras walks without boundary weights~\cite{bousquet2010walks}. There, the full-orbit sum yields no useful information, and instead the half-orbit sum alone is sufficient to solve the model.

\subsection{The symmetry group}

The \emph{symmetry group} of a lattice path model is the set of birational transformations of $(x,y)$ which leave the kernel $K(x,y)$ unchanged~\cite{bousquet2010walks}. This group is generated by the pair
\begin{equation}
\phi :(x,y)\mapsto\left(\frac{B_{-1}(y)}{xB_1(y)},y\right) \  \ \text{and}\  \ \psi:(x,y)\mapsto\left(x,\frac{A_{-1}(x)}{yA_1(x)}\right).
\end{equation}
For reverse Kreweras walks, these are
\begin{equation}
\phi :(x,y)\mapsto(\olxy,y) \  \ \text{and}\  \ \psi:(x,y)\mapsto(x,\olxy).
\end{equation}
The resulting group is isomorphic to $D_3$:
\begin{equation}
(x,y),\ \ (\olxy,y),\ \ (y,\olxy),\ \ (y,x),\ \ (\olxy,x),\ \ (x,\olxy).\label{symmetries}
\end{equation}

\subsection{Full-orbit sum}

Applying all these symmetries to the functional equation~\eqref{1x1x} yields six equations. For simplicity, we write these equations in a matrix form:
\begin{equation}
 K(x,y)\mathbf{Q}= \mathbf{MV}+ \mathbf{C}\label{K(x,y)Q=MV+C}.
\end{equation}
Here $\mathbf{Q}$ is the column vector of all transformed $Q(x,y)$ and $\mathbf{V}$ is the transformed line boundary terms
\begin{equation}\label{eqn:rev_Q_vec}
\mathbf{Q}=\begin{pmatrix}
Q(x,y) \\
Q(\olxy,y)\\
Q(y,\olxy)\\
Q(y,x)\\
Q(\olxy,x)\\
Q(x,\olxy)
\end{pmatrix}, \hspace{1cm}
\mathbf{V}=\begin{pmatrix}
Q(x,0) \\
Q(0,y)\\
Q(\olxy,0)\\
Q(0,\olxy)\\
Q(0,x)\\
Q(y,0)
\end{pmatrix},
\end{equation}
and $\mathbf{M}$ is the coefficient matrix
\begin{equation}\label{eqn:rev_M_matrix}
\mathbf{M}=\begin{pmatrix}
    A'(x,y) & B'(x,y)  &0  & 0&0&0\\   
    0 &  B'(\olxy,y)& A'(\olxy,y)&0&0&0\\   
    0 & 0 & 0 & B'(y,\olxy) & 0 & A'(y,\olxy)\\
	0 & 0 & 0 & 0 & B'(y,x) & A'(y,x)\\
    0 & 0 & A'(\olxy,x) & 0 & B'(\olxy,x) & 0\\
    A'(x,\olxy) & 0 & 0 & B'(x,\olxy) & 0 & 0
\end{pmatrix}.
\end{equation}
Here $A'(x,y) = \frac1a(a-1-ta\olxy)$ and $B'(x,y) = \frac1b(b-1-tb\olxy)$. The column vector $\mathbf{C}$ contains the point boundary terms $Q(0,0)$ and some known terms in $\mathbb{R}((x,y,t))$:
\begin{equation}
\mathbf{C}=\frac{1}{abc}\begin{pmatrix}
{\left[-ab+ac+bc-abc+tabc\olxy\right]Q(0,0) +a b} \\
{\left[-ab+ac+bc-abc+tabcx\right]Q(0,0) +a b}\\
{\left[-ab+ac+bc-abc+tabcx\right]Q(0,0) +a b}\\
{\left[-ab+ac+bc-abc+tabc\olxy\right]Q(0,0) +a b}\\
{\left[-ab+ac+bc-abc+tabcy\right]Q(0,0) +a b}\\
{\left[-ab+ac+bc-abc+tabcy\right]Q(0,0) +a b}
\end{pmatrix}.
\end{equation}
It is straightforward to show that the determinant of $\mathbf{M}$ equals zero. Thus, there exists a linear combination of these six equations that cancels all variables in $\mathbf{V}$. We choose a vector $\mathbf{N}$ that spans the nullspace:
\begin{equation}
\mathbf{N}=
\begin{pmatrix}
-y(1-b+tbx)(1-a+tay) \\
\olx(1-a+tay)(tb+xy-bxy) \\
-\olx(1-b+tby)(ta+xy-axy) \\
y(1-a+tax)(1-b+tby) \\
-\olx(1-a+tax)(tb+xy-bxy) \\
-\olx(1-b+tbx)(ta+xy-axy)
\end{pmatrix}^\mathsf{T}
\end{equation}
Multiplying $\mathbf{N}$ to both sides of~\eqref{K(x,y)Q=MV+C}, we have
\begin{equation}
K(x,y)\mathbf{N} \mathbf{Q}=\mathbf{N MV}+\mathbf{N C}=\mathbf{N C}\label{Ab}
\end{equation}
since $\mathbf{N M}=\mathbf{0}$. The equation~\eqref{Ab} is called the \emph{full-orbit sum}.

We also have $\mathbf{NC} = \mathbf{0}$ in~\eqref{Ab}. This happens in all non-interacting algebraic models~\cite{bousquet2010walks}, including Kreweras and reverse Kreweras walks. In fact it only happens for algebraic models since for all walks with transcendental generating functions, $\mathbf{N C}\neq \mathbf{0}$~\cite{mishna2009classifying}. Note that the full-orbit sum is redundant in the unweighted case because $Q(x,y)=Q(y,x)$; the lack of such symmetry here means that we can extract useful information from~\eqref{Ab}.

Next, divide $K(x,y)$ on both sides and then extract the $[y^0]$ term of~\eqref{Ab}. Some new boundary terms will appear when performing the extraction.

By geometric properties, we can eliminate some of them and find an equation of the form
\begin{equation}\label{simplified orbit sum}
[y^0]\mathbf{NQ} = 0 = \alpha + \alpha_{0,0}Q(0,0) + \alpha_{0,x}Q(0,x) + \alpha_{x,0}Q(x,0) + \alpha^\mathrm{d}_0 Q^\mathrm{d}_0(\olx) + \alpha^\mathrm{d}_1 Q^\mathrm{d}_1(\olx)
\end{equation}
where the $\alpha$ coefficients are Laurent polynomials in $x$ and $t$.

For an equation with this form, the terms with positive and negative powers of $x$ can be naturally separated. We then take the $[x^>]$ and $[x^<]$ parts. After some simplifications by boundary conditions, we obtain the following.

\begin{lemma}\label{lem:revkrew_fos_pos_and_neg}
The $[x^>]$ and $[x^<]$ parts of \eqref{simplified orbit sum} can be written as
\begin{equation}\label{(x,0) and (0,x)}
[x^>y^0]\mathbf{NQ} = 0 = \beta + \beta_{0,0}Q(0,0) + \beta_{0,1}Q_{0,1} + \beta_{1,0}Q_{1,0} + \beta_{x,0}Q(x,0) + \beta_{0,x}Q(0,x)
\end{equation}
and
\begin{equation}\label{1/x}
[x^<y^0]\mathbf{NQ} = 0 = \gamma + \gamma_{0,0}Q(0,0) + \gamma_{0,1}Q_{0,1} + \gamma_{1,0}Q_{1,0} + \gamma^\mathrm{d}_0 Q^\mathrm{d}_0(\olx) + \gamma^\mathrm{d}_1 Q^\mathrm{d}_1(\olx),
\end{equation}
where the $\beta$ and $\gamma$ coefficients are Laurent polynomials in $x$ and $t$.
\end{lemma}

When $a=b$,~\eqref{(x,0) and (0,x)} and~\eqref{1/x} give $Q(x,0) = Q(0, x)$ and $0 = 0$. This step is omitted when solving the non-interacting or equally-interacting cases~\cite{beaton_exact_2019,bousquet2005walks} since the geometric symmetry is obvious.
 
\subsection{Half-orbit sum}

We have now obtained some geometric symmetries of reverse Kreweras walks from the full-orbit sum. However, we lost the information contained in the kernel, since the RHS of the full-orbit sum was $0$. We will now ``regain'' this information by taking a half-orbit sum.

Recall the six equations obtained by applying the symmetry group. We rewrite the matrix equation~\eqref{K(x,y)Q=MV+C} slightly differently:
\begin{equation}
K(x,y)\mathbf{Q}=\mathbf{M}_2 \mathbf{V}_2+\mathbf{C}_2\label{half}
\end{equation}
Unlike the full-orbit sum, $Q(x,0)$ and $Q(0,x)$ are not included in the vector $\mathbf{V}_2$. Instead, $\mathbf{V}_2$ is
\begin{equation}
V_2=\begin{pmatrix}
Q(0,y)\\
Q(\olxy,0)\\
Q(0,\olxy)\\
Q(y,0)\label{rkh}
\end{pmatrix}
\end{equation}
The terms $Q(x,0)$ and $Q(0,x)$ are treated as constant and we put them in $\mathbf{C}_2$. $\mathbf{M}_2$ contains the corresponding coefficients of $\mathbf{V}_2$ in $\mathbf{M}$. So $\mathbf{M}_2$ is a $6\times 4$ rectangular matrix. The dimension of the null-space is $2$. Two orthogonal basis vectors span $\mathbf{M_2}^\mathsf{T}$. We may choose either of them in the following derivation. The vector we choose is 

\begin{equation}
\mathbf{N}_2 = \begin{pmatrix}
-x(1-b+tbx)(1-a+tay) \\
\oly(1-a+tay)(tb+xy-bxy) \\
0\\
0\\
-\oly(1-a+tax)(tb+xy-bxy)\\
0
\end{pmatrix}^\mathsf{T}.
\end{equation}
Notice that $\mathbf{N}_2$ has three zeros and three non-zeros, corresponding to a half-orbit sum.

Then, multiply $\mathbf{N}_2$ on~\eqref{half} and divide by $K(x,y)$. We find 
\begin{equation}
\mathbf{N}_2\mathbf{Q}=\frac{\mathbf{N}_2\mathbf{C_2}}{K(x,y)},\label{hos}
\end{equation}
since $\mathbf{N}_2\mathbf{M}_2=0$. The equation~\eqref{hos} is called the \emph{half-orbit sum}.

As we did for the full-orbit sum case, we extract the $[y^0]$ of~\eqref{hos} to get an equation without the variable $y$. We can again use boundary conditions to simplify this equation, eventually finding an equation of the form
\begin{equation}\label{rev_halfOS_lhs_y0}
[y^0]\mathbf{N_2Q} = \delta + \delta_{0,0}Q(0,0) + \delta_{x,0}Q(x,0) + \delta_{0,x}Q(0,x) + \delta^\mathrm{d}_0 Q^\mathrm{d}_0(\olx) + \delta^\mathrm{d}_1 Q^\mathrm{d}_1(\olx)
\end{equation}
where the $\delta$ coefficients are polynomials in $x$ and $t$.

Now, consider the RHS.
The RHS contains two parts, $\mathbf{N}_2\mathbf{C}_2$ and $1/K(x,y)$. First consider $\mathbf{N}_2\mathbf{C}_2$. It is a linear equation of $Q(x,0)$, $Q(0,x)$ and $Q(0,0)$ with coefficients in $\mathbb{R}^r(x,t)((y))$. It can be treated as a function of $y$ whose coefficients are linear terms of $Q(x,0)$, $Q(0,x)$, $Q(0,0)$ and $\mathbb{R}^r(x,t)$, and by examination we find  $\mathbf{N}_2\mathbf{C}_2$ only contains $y^{-1},y^0,y^1$ terms. So it can be written as 
\begin{equation}
\mathbf{N}_2\mathbf{C}_2 = u_{-1} y^{-1} + u_{0} y^{0} +  u_{1}y^{1}.
\end{equation}

Then consider $\frac{1}{K(x,y)}$. It can be expanded as a power series in $t$ with coefficients that are polynomials in $x,\olx,y,\oly$. By collecting all terms with a given power of $y$, we can define 
$K_i = [y^i]\frac{1}{K(x,y)}$.

It follows that the $[y^0]$ term of~\eqref{hos} is 
\begin{align}
[y^0]\left(\frac{\mathbf{N}_2\mathbf{C}_2}{K(x,y)}\right) &= u_{-1} K_1+ u_{0} K_0 + u_{1} K_{-1} \\
 &= \epsilon + \epsilon_{0,0} Q(0,0) + \epsilon_{x,0}Q(x,0) + \epsilon_{0,x}Q(0,x)\label{rev_halfOS_rhs_y0}
\end{align}
where the $\epsilon$ coefficients are Laurent polynomials in $x,t$ and $K_1,K_0,K_{-1}$.

In the next section, we demonstrate how to compute $K_i$.

\subsection{The roots of the kernel}

The following lemma will be useful.

\begin{lemma}[Lemma 7 in~\cite{bousquet2010walks}]\label{l1}
For a quadratic equation $K(x,y)=1-t S(x,y)$, where $S(x,y)$ is defined as per~\eqref{S}, we have:
$1/K(x,y)$ is a formal power series of $t$ with polynomial coefficients in $x,\olx,y,\oly$, and
\begin{equation}
\frac{1}{K(x,y)}=\frac{1}{\sqrt{\Delta(x)}}\left(\frac{1}{1-\oly Y_0(x)}+\frac{1}{1-y/Y_1(x)}-1\right)\label{canonical}
\end{equation}
where
\begin{equation}
\Delta(x)=(1-tA_0(x))^2-4t^2A_{-1}(x)A_1(x)
\end{equation}
and
\begin{equation}
Y_0=\frac{1-tA_0(x)-\sqrt{\Delta(x)}}{2tA_1(x)}, \    \ Y_1=\frac{1-tA_0(x)+\sqrt{\Delta(x)}}{2tA_1(x)}.
\end{equation}
Thus
\begin{equation}
[y^j]\frac{1}{K(x,y)}= K_j = \begin{cases}
\frac{Y_0(x)^{-j}}{\sqrt\Delta(x)} & \text{if } j\leq 0\\
\frac{Y_1(x)^{-j}}{\sqrt\Delta(x)} & \text{if } j\geq 0\label{ccf}.
\end{cases}
\end{equation}
\end{lemma}
\begin{proof}
Factoring $K(x,y)$, we have
\begin{equation}
K(x,y)= t A_1(x)(y-Y_0)(y-Y_1)=\frac{tA_1(x)}{Y_1y}(1-\oly Y_0)(1-y/Y_1)\label{fa}
\end{equation}
Rearranging yields~\eqref{canonical}. Since $Y_0(x)$ has valuation $1$ in $t$ and $Y_1(x)$ has valuation $-1$ in $t$, the RHS of~\eqref{ccf} is a formal power series in $t$.
\end{proof}

Now equating~\eqref{rev_halfOS_lhs_y0} and~\eqref{rev_halfOS_rhs_y0}, using~\eqref{(x,0) and (0,x)} and~\eqref{1/x} to eliminate $Q(0,x)$ and $Q^\mathrm{d}_1(\olx)$,  and substituting the explicit expressions for $K_i$, gives 
\begin{multline}\label{eqn:rk_halforbitsum_y0}
\mu_{x,0}Q(x,0) + \nu^\mathrm{d}_0\sqrt{\Delta}Q^\mathrm{d}_0(\olx) = (\mu+\nu\sqrt{\Delta}) + (\mu_{0,0}+\nu_{0,0}\sqrt{\Delta})Q(0,0) \\ + (\mu_{0,1}+\nu_{0,1}\sqrt{\Delta})Q_{0,1} + (\mu_{1,0}+\nu_{1,0}\sqrt{\Delta})Q_{1,0},
\end{multline}
where the $\mu$ and $\nu$ are all \emph{polynomials}, and
\begin{equation}
\sqrt{\Delta}=\sqrt{t^2\olx \left(x^3-4\right)-2 t x+1}.
\end{equation}

Unfortunately~\eqref{eqn:rk_halforbitsum_y0} still cannot be separated directly, as $\sqrt\Delta$ is a Laurent series of $x$ and $Q^d_0(\olx)$ is a series of $\olx$ with unknown coefficients. Then the $[x^>]$ of $\sqrt{\Delta(x)}Q^d_0(\olx)$ involves an infinite number of unknowns. We need to introduce the canonical factorisation of $\Delta$, which provides a way to deal with $\sqrt{\Delta}$.

\subsection{The canonical factorisation}

\begin{lemma}[The canonical factorisation of $\Delta$~\cite{bousquet2010walks}]\label{lem:canonical_factorisation}
We have
\begin{equation}
\Delta(x)=(1-tA_0(x))^2-4t^2A_{-1}(x)A_1(x)
\end{equation}
If $\Delta(x)$ has valuation $-\delta$, and degree $d$ in x, then $\Delta(x)=0$ has $\delta+d$ roots. Exactly $\delta$ of them (say $X_1,\dots,X_\delta$) are finite (actually vanish) when $t=0$ and the remaining $d$ roots ($X_{\delta+1},\dots,X_{\delta+d}$) have negative valuation in $t$ and thus diverge when $t=0$. 

Then $\Delta$ can be factored as:
\begin{equation}
\Delta(x)=\Delta_0\Delta_{-}(\olx)\Delta_{+}(x)\label{fac}
\end{equation}
with 
\begin{align}
\Delta_{-} &\equiv \Delta_{-}(\olx,t) = \Pi_{i=1}^{\delta}(1-X_i/x) \label{d1}\\
\Delta_{+} &\equiv \Delta_{+}(\olx,t) = \Pi_{i=1+\delta}^{\delta+d}(1-x/X_i) \label{d2}\\
\Delta_{0} &\equiv \Delta_{0}(t) = (-1)^{\delta}\frac{[\olx^{\delta}]\Delta(x)}{\Pi_{i=1}^{\delta}X_i}=(-1)^d[x^d]\Delta(x)\Pi_{i=\delta+1}^{\delta+d}X_i. \label{d3}
\end{align}
\end{lemma}
\begin{proof}
We can factor $\Delta$ by its roots. For roots that are finite  at $t=0$, we factor them as $x(1-X_i\olx)$. For roots that diverge around $t=0$, we factor them as $X_i(1-x/X_i)$. This leads to~\eqref{fac}.
\end{proof}
We can check that $\Delta_{0},\Delta_{+},\Delta_{-}$ are formal power series in $t$ with constant term $1$. We will use this property in following derivation. This factorisation is a canonical factorisation -- see Gessel's work~\cite{gessel1980factorization} for more information.

In our case, $\Delta$ has one root which is a power series in $t$ and two which are not:
\begin{align}
X_1 &= 4 t^2 + 32 t^5 + 448 t^8 + \bigO(t^{11}) \\
X_2 &= t^{-1} + 2 t^{1/2} - 2 t^2 + 5 t^{7/2} - 16 t^5 + \textstyle\frac{231}{4} t^{13/2} - 
 224 t^8 + \frac{7293}{8} t^{19/2} + \bigO(t^{11}) \\
X_3 &= t^{-1} - 2 t^{1/2} - 2 t^2 - 5 t^{7/2} - 16 t^5 - \textstyle\frac{231}{4} t^{13/2} - 
 224 t^8 - \frac{7293}{8} t^{19/2} + \bigO(t^{11}) .
\end{align}
The canonical factorisation is then
\begin{align}
\Delta_0 &= t^2X_2 X_3\\
\Delta_{+} &= (1-x/X_2)(1-x/X_3)\\
\Delta_{-} &= 1-X_1\olx
\end{align}
so that
\begin{align}
\frac{1}{\sqrt{\Delta_+}} &= 1 + tx + t^2x^2  + t^3x^3  + t^4(6 x + x^4) + \bigO(t^5) \\
\sqrt{\Delta_0\Delta_-} &= 1 - 2t^2\olx - 4t^3 - 2t^4\olx^2 + \bigO(t^5)
\end{align}

\subsection{The algebraic solution of reverse Kreweras walks}\label{sec:alg_soln_rk}

We now take~\eqref{eqn:rk_halforbitsum_y0} and divide by $\sqrt{\Delta_+}$. Each $\mu$ term, divided by $\sqrt{\Delta_+}$, produces only non-negative powers of $x$. Each $\nu$ term, now multiplied by $\sqrt{\Delta_0\Delta_-}$ only, produces only finitely many non-negative powers of $x$. It is thus possible to compute the positive and negative parts with respect to $x$. Doing so, rearranging a bit and sending $x\mapsto\olx$ in the second equation gives the following.
\begin{lemma}\label{lem:revkrew_two_kernel_equations}
\begin{align}
\sigma_{x,0}P_{x,0}Q(x,0) &= \sigma + \sigma_{0,0}Q(0,0) + \sigma_{0,1}Q_{0,1} + \sigma_{1,0}Q_{1,0} \label{eqn:qx0_pos} \\
\tau^\mathrm{d}_0P^\mathrm{d}_0Q^\mathrm{d}_0(x) &= \tau + \tau_{0,0}Q(0,0) + \tau_{0,1}Q_{0,1} + \tau_{1,0}Q_{1,0},\label{eqn:qdx_neg}
\end{align}
where
\begin{align}
P_{x,0} &= (a^2 t^2 + x - a x - a t x^2 + a^2 t x^2) (2 a b t^2 + 2 x - a x - 
   b x - a t x^2 - b t x^2 + 2 a b t x^2) \\
P^\mathrm{d}_0 &= (-a t + a^2 t + x - a x + a^2 t^2 x^2) (-b t + b^2 t + x - b x + 
   b^2 t^2 x^2)
\end{align}
and the $\tau$ and $\sigma$ coefficients are algebraic functions of $t$ and $x$ (and $(a,b,c)$).
\end{lemma}

Now observe that~\eqref{eqn:qx0_pos} and~\eqref{eqn:qdx_neg} are in the form of kernel equations. The unknown functions $Q(x,0)$ and $Q^\mathrm{d}_0(x)$ are formal power series of $t$ with polynomial coefficients in $x$.

\begin{lemma}\label{lem:revkrew_twokernels_roots}
The kernel $P_{x,0}$ has two quadratic factors. The roots are 
\begin{align}
x_{1,2} &= \frac{\mp\sqrt{-4 a^4 t^3+4 a^3 t^3+a^2-2 a+1}+a-1}{2at(a-1)}\\
x_{3,4} &= \frac{\mp\sqrt{(a+b-2)^2-8 a b t^3(2ab-a-b)}+a+b-2}{2t(2ab-a-b)}
\end{align}
If $a>1$ (resp.~$0<a<1$) then $x_1$ (resp.~$x_2$) is a formal power series of $t$ and converges as $t\to0$. Similarly, if $a+b>2$ (resp.~$0<a+b<2$) then $x_3$ (resp.~$x_4$) is a formal power series of $t$ and converges as $t\to0$. 

The kernel $P^\mathrm{d}_0$ also has two quadratic factors, leading to two pairs of roots
\begin{align}
x_{5,6} &= \frac{\mp\sqrt{-4 a^4 t^3+4 a^3 t^3+a^2-2 a+1}+a-1}{2 a^2 t^2}\\
x_{7,8} &= \frac{\mp\sqrt{-4 b^4 t^3+4 b^3 t^3+b^2-2 b+1}+b-1}{2 b^2 t^2}.
\end{align}
If $a>1$ (resp.~$0<a<1$) then $x_5$ (resp.~$x_{6}$) is a formal power series of $t$ and converges as $t\to0$. Similarly, if $b>1$ (resp.~$0<b<1$) then $x_{7}$ (resp.~$x_{8}$) is a formal power series of $t$ and converges as $t\to0$.
\end{lemma}

\begin{proof}[Proof of \cref{1111}]
For any $a,b>0$ with $a\neq1,b\neq1$ and $a+b\neq2$, by substituting appropriate roots into~\eqref{eqn:qx0_pos} and~\eqref{eqn:qdx_neg}, we get a set of four equations of the form
\begin{equation}
0 = \zeta^{(i)} + \zeta^{(i)}_{0,0}Q(0,0) + \zeta^{(i)}_{0,1}Q_{0,1} + \zeta^{(i)}_{1,0}Q_{1,0} \equiv H_i,
\end{equation}
where $i$ indicates which $x_i$ root has been used in~\eqref{eqn:qx0_pos} or~\eqref{eqn:qdx_neg}.

For simplicity assume $a,b>1$, so that $i\in\{1,3,5,7\}$. The final question, then, is which combination(s) of these equations, if any, give a solution. That is, for which $i,j,k$ is the determinant
\begin{equation}
D_{i,j,k} = \zeta^{(i)}_{0,0} \zeta^{(j)}_{0,1} \zeta^{(k)}_{1,0} - \zeta^{(i)}_{0,0} \zeta^{(j)}_{1,0} \zeta^{(k)}_{0,1} - \zeta^{(i)}_{0,1} \zeta^{(j)}_{0,0} \zeta^{(k)}_{1,0} + \zeta^{(i)}_{0,1} \zeta^{(j)}_{1,0} \zeta^{(k)}_{0,0} + \zeta^{(i)}_{1,0} \zeta^{(j)}_{0,0} \zeta^{(k)}_{0,1} - \zeta^{(i)}_{1,0} \zeta^{(j)}_{0,1} \zeta^{(k)}_{0,0}
\end{equation}
non-zero? Curiously, it appears that $D_{1,3,5}=0$ while
\begin{align}
D_{1,3,7} &= \frac{16a^6b^4c^2(a-1)(a-2)(a-b)^2(ab-1)t^{10}}{a+b-2} + \bigO(t^{11}) \\
D_{1,5,7} &= -8a^8b^3c^2(a-1)^2(b-1)^2(a-b)^2(ab-a+1)t^{10} + \bigO(t^{11}) \\
D_{3,5,7} &= -\frac{16a^7b^4c^2(a-1)^2(b-1)^2(a-b)^2(ab-1)t^{10}}{a+b-2} + \bigO(t^{11})
\end{align}
(These determinants depend on the choice of $\sigma_{x,0}$ and $\tau^\mathrm{d}_0$ from~\eqref{eqn:qx0_pos} and~\eqref{eqn:qdx_neg}. What is important is whether they are 0 or not.)

Choosing one of the valid combinations gives algebraic solutions to $Q(0,0)$, $Q_{0,1}$ and $Q_{1,0}$. By back-substitution into~\eqref{eqn:qx0_pos} we then get the solution to $Q(x,0)$, and then by symmetry $Q(0,x)$ (and hence $Q(0,y)$). Finally, the original equation~\eqref{1x1x} yields $Q(x,y)$. (To save space we will not explicitly write down the expression.) Because all involved terms are algebraic, $Q(x,y)$ is too, specifically over $\mathbb{C}(t,x,y,a,b,c)$.

If $a<1$ or $b<1$ then the roots which are not power series in $t$ can be swapped out as required. In all cases the determinants which were non-zero remain so.
\end{proof}

\subsection{Some special cases}

There are some special cases where things appear to break down, namely $a=1$, $b=1$ and $a+b=2$. When $a+b=2$ the second factor in $P_{x,0}$ loses its $x$ term. As a result, neither $x_3$ nor $x_4$ are valid power series roots of $P_{x,0}$, and we lose one of our equations in $Q(0,0)$, $Q_{0,1}$ and $Q_{1,0}$. However, if $(a,b) \neq (1,1)$, the remaining three equations are still valid, and the solution still emerges.

When $b=1$, the system simplifies dramatically. One finds that the coefficient $\sigma_{1,0}$ of $Q_{1,0}$ vanishes. The two equations then obtained (either $H_1$ and $H_3$, or $H_2$ and $H_4$, depending on $a$) are linearly independent, and the solution follows. Naturally the $a=1$ case is then just a reflection.

In addition to using the $[x^>]$ and $[x^<]$ parts of~\eqref{eqn:rk_halforbitsum_y0}, the $[x^0]$ part also provides an equation (say, $H_0$) with unknowns $Q(0,0),Q_{1,0},Q_{0,1}$. This is not in a kernel form (it has no dependence on $x$ or $y$), but it can be combined with the other equations discussed above. It turns out that the equation sets formed by $\{H_0,H_1,H_{7}\}$ or $\{H_0,H_5,H_7\}$ have non-zero determinants, and can thus also be used to generate the solutions.

\section{Kreweras walks}\label{sec:krew}

We now turn our attention to Kreweras walks, and attempt to obtain the solution using a similar method. We still use the algebraic kernel method, but the process is different because some symmetries have changed. 
\begin{itemize}
\item First, we recall the symmetry group of the kernel and use it to take the full-orbit sum. We then extract the $[x^{>}y^0]$ part (\cref{lem:krew_fos_pos_part}).
\item Secondly, we take a half-orbit sum and again take the $[y^0]$ terms. We use the equations obtained from the full-orbit sum to eliminate certain boundary terms, and then once again take the positive and negative parts with respect to $x$. This yields two equations which have a kernel-like form (\cref{lem:krew_two_kernel_equations}).
\item Thirdly, we cancel the kernels of these two equations (see \cref{lem:krew_twokernels_roots}) and find a set of linear equations. 
\item Finally, we substitute the result of $Q(0,0)$ into the equations obtained in the previous step and solve the linear equation set.
\end{itemize}

The first two steps are the same as for reverse Kreweras walks. The difference is that we are unable to directly solve the problem by the third step. So we need some extra work. We will show how to do this in the following section. Note that, as with reverse Kreweras walks, the unweighted case can be solved with the half-orbit sum alone~\cite{bousquet2010walks}.

For simplicity some notation from the previous sections will be reused -- no definitions carry over unless otherwise indicated.

\subsection{The functional equation and the symmetry group}

For Kreweras walks, the allowed steps are $\{\mathrm{NE, S, W}\}$. Recall the functional equation~\eqref{eqn:krew_func_eqn}:
\begin{multline}
\left(1-t(xy+\olx+\oly)\right)Q(x,y)=\frac1c + \frac1a(a-1 - ta\oly)Q(x,0) + \frac1b(b - 1 - tb\olx)Q(0,y) \\
+\frac{1}{abc}(ac+bc-ab-abc)Q(0,0).
\end{multline}
The symmetry group for Kreweras walks is the same as for reverse Kreweras walks.

\subsection{Full-orbit sum}

We again apply the symmetries to the functional equation and write them in a matrix form
\begin{equation}\label{Kr1}
 K(x,y)\mathbf{Q}=\mathbf{MV}+\mathbf{C}.
 \end{equation}
Here $\mathbf{Q}$ and $\mathbf{V}$ are as per~\eqref{eqn:rev_Q_vec}, while $\mathbf{C}$ is just the constant vector of
\begin{equation}
\frac{(-ab+ac+bc-abc)Q(0,0)}{abc} + \frac1c,
\end{equation}
and $\mathbf{M}$ is as per~\eqref{eqn:rev_M_matrix} but now with $A'(x,y)=\frac1a(a-1-ta\oly)$ and $B'(x,y)=\frac1b(b-1-tb\olx)$.

For Kreweras walks, we still have $\det(\mathbf M)=0$, which means we can find a linear combination of these equations that cancels all variables in $\mathbf{V}$. The nullspace is spanned by the vector
\begin{equation}
\mathbf{N}=\begin{pmatrix}
(a t + x - a x) (b t + y - b y) (1 - a + a t x y) (1 - b + b t x y)\\
-\olx(a t + x - a x) (b t + x - b x) (b t + y - b y) (1 - a + a t x y)\\
\olx(a t + x - a x) (b t + x - b x) (a t + y - a y) (1 - b + b t x y)\\
-(b t + x - b x) (a t + y - a y) (1 - a + a t x y) (1 - b + b t x y)\\
\oly(b t + x - b x) (a t + y - a y) (b t + y - b y) (1 - a + a t x y)\\
-\oly(a t + x - a x) (a t + y - a y) (b t + y - b y) (1 - b + b t x y)
\end{pmatrix}^\mathsf{T}
\end{equation}
As with reverse Kreweras, left-multiplying by $\mathbf{N}$ gives 
\begin{equation}\label{kreweras_multiply_N}
K(x,y)\mathbf{N}\mathbf{Q} = \mathbf{N}\mathbf{C}. 
\end{equation}
But now the difference becomes apparent. For reverse Kreweras, we had $\mathbf{NC}=0$ (also true for unweighted  Kreweras walks), but here we have 
\begin{equation}
\mathbf{NC}=\frac{t^3 \olxy(a-b) (x-y) \left(x^2 y-1\right) \left(x y^2-1\right) \big[ab- (a b - a c - b c + a b c)Q(0,0)\big]}{c}.\label{NCK}
\end{equation}
This is more complicated than reverse Kreweras (except, of course, when $a=b$), but it is still in a form that we can deal with. Analogously to the half-orbit sum in the reverse Kreweras case, we divide~\eqref{kreweras_multiply_N} by the kernel and take the $[y^0]$ part. After simplification by boundary relations, the left hand side is a linear equation of the form
\begin{equation}\label{kreweras simplified orbit sum}
[y^0]\mathbf{NQ} = \alpha + \alpha_{0,0}Q(0,0) + \alpha_{0,x}Q(0,x) + \alpha_{x,0}Q(x,0) + \alpha^\mathrm{d}_0 Q^\mathrm{d}_0(\olx) + \alpha^\mathrm{d}_1 Q^\mathrm{d}_1(\olx)
\end{equation}
where the $\alpha$ coefficients are Laurent polynomials in $x$ and $t$. For the RHS, since $\mathbf{NC}$ does not contain $Q(0,y)$, it can be regarded as a polynomial of $y$, namely
\begin{equation}
\mathbf{NC} = v_ {-1} y^{-1} + v_ {0} y^{0} +  v_ {1} y^{1} +  v_ {2} y^{2} + v_3y^3.
\end{equation}

As with reverse Kreweras walks we can also define $K_i=[y^i]\frac{1}{K(x,y)}$.
It follows that
\begin{align}
[y^0]\left(\frac{\mathbf{NC}}{K(x,y)}\right) &= v_ {-1} K_1 + v_{0} K_0 +  v_{1} K_{-1} +  v_{2}K_{-2} + v_{3}K_{-3} \\
&= \eta + \eta_{0,0}Q(0,0)\label{eqn:krew_fullOS_rhs_y0}
\end{align}
where the $\eta$ coefficients are Laurent polynomials in $x,t$ and the $K_i$.

The main result obtained from the full-orbit sum is the following.
\begin{lemma}\label{lem:krew_fos_pos_part}
The $[x^>]$ part of~\eqref{kreweras simplified orbit sum} can be written as
\begin{equation}\label{krew_fullOS_lhs_xposy0}
[x^>y^0]\mathbf{NQ} = \beta + \beta_{x,0}Q(x,0) + \beta_{0,x}Q(0,x) + \beta_{0,0}Q(0,0) + \beta_{1,0}Q_{1,0} + \beta_{2,0}Q_{2,0} + \beta_{3,0} Q_{3,0}
\end{equation}
where the $\beta$ coefficients are Laurent polynomials in $t,x$. The $[x^>]$ part of~\eqref{eqn:krew_fullOS_rhs_y0} can be written as
\begin{equation}\label{krew_fullOS_rhs_xposy0}
[x^>y^0]\left(\frac{\mathbf{NC}}{K(x,y)}\right) = \theta + \theta_{0,0}Q(0,0).
\end{equation}
where $\theta$ and $\theta_{0,0}$ are D-finite. That is, the vector space over $\mathbb{C}(t,x,a,b,c)$ spanned by all partial derivatives of $\theta$ (with respect to $t$ and $x$ only) has finite dimension, and likewise for $\theta_{0,0}$.
\end{lemma}

\begin{proof}
The former follows immediately from the form of~\eqref{kreweras simplified orbit sum}. For the latter, note that the $\eta$ coefficients in~\eqref{eqn:krew_fullOS_rhs_y0} are algebraic functions of $t$ and $x$, being rational functions of $x,t$ and $\sqrt{\Delta}$, where
\begin{equation}
\Delta = (1-t\olx)^2-4t^2x.
\end{equation}
As the positive parts of algebraic functions, we can write $\theta$ and $\theta_{0,0}$ into series form and do know that they are D-finite (see, for example,~\cite[Theorem 3.7]{lipshitz1989d}). We have no reason to believe that they are algebraic.
\end{proof}

Equating~\eqref{krew_fullOS_lhs_xposy0} and~\eqref{krew_fullOS_rhs_xposy0}, we obtain an equation relating $Q(x,0), Q(0,x)$ and several $x$-independent unknowns -- the equivalent of~\eqref{(x,0) and (0,x)}. Unlike reverse Kreweras walks, it will not be necessary to also take the $[x^<y^0]$ part of the full-orbit sum.

\subsection{Half-orbit sum}

Following the same process as for reverse Kreweras walks, we now take the half-orbit sum. The process is nearly the same. We will also have equations similar to~\eqref{half}, and~\eqref{rkh} in this case. The half-orbit sum is
\begin{equation}
K(x,y)\mathbf{Q}=\mathbf{M}_2 \mathbf{V}_2+\mathbf{C}_2\label{half2}
\end{equation}
where
\begin{equation}
\mathbf{V}_2=\begin{pmatrix}
Q(0,y)\\
Q(\olxy,0)\\
Q(0,\olxy)\\
Q(y,0)\label{rkh2}
\end{pmatrix}
\end{equation}
We have two vectors spanning the nullspace. The one we choose is 
\begin{equation}
\mathbf{N_2}=
\begin{pmatrix}
(a t + x - a x) (1 - b + b t x y)\\
-\olx(a t + x - a x) (b t + x - b x)\\
0\\
0\\
\oly(b t + x - b x) (a t + y - a y)\\
0
\end{pmatrix}
\end{equation}
and then
\begin{equation}
\mathbf{N}_2\mathbf{Q}_2=\frac{\mathbf{N}_2\mathbf{C}_2}{K(x,y)}.\label{kf}
\end{equation}
We take the $[y^0]$ term of~\eqref{kf} analogously to the reverse Kreweras case, and end up with the equivalents of~\eqref{rev_halfOS_lhs_y0} and~\eqref{rev_halfOS_rhs_y0}:
\begin{align}
[y^0]\mathbf{N_2Q} &= \delta + \delta_{0,0}Q(0,0) + \delta_{x,0}Q(x,0) + \delta_{0,x}Q(0,x) + \delta^\mathrm{d}_0 Q^\mathrm{d}_0(\olx) \label{eqn:krew_halfOS_lhs_y0}\\
[y^0]\left(\frac{\mathbf{N}_2\mathbf{C}_2}{K(x,y)}\right) &= \epsilon + \epsilon_{0,0} Q(0,0) + \epsilon_{x,0}Q(x,0) + \epsilon_{0,x}Q(0,x). \label{eqn:krew_halfOS_rhs_y0}
\end{align}
As with reverse Kreweras walks, the $\delta$ coefficients are Laurent polynomials in $t,x$ while the $\epsilon$ coefficients are Laurent polynomials in $t,x$ and the $K_i$. (The absence of $Q^\mathrm{d}_1(\olx)$ in~\eqref{eqn:krew_halfOS_lhs_y0} is why we did not need to take the $[x^<y^0]$ part of the full-orbit sum.)

We then equate~\eqref{eqn:krew_halfOS_lhs_y0} and~\eqref{eqn:krew_halfOS_rhs_y0} and use~\eqref{krew_fullOS_lhs_xposy0} and~\eqref{krew_fullOS_rhs_xposy0} to eliminate $Q(0,x)$. This leads to an equation of the form
\begin{multline}\label{eqn:krew_Qx0_and_Qd1x}
\mu_{x,0}Q(x,0) + \nu^\mathrm{d}_0\sqrt{\Delta}Q^\mathrm{d}_0(\olx) = (\hat\mu+\hat\nu\sqrt{\Delta}) + (\hat\mu_{0,0}+\hat\nu_{0,0}\sqrt{\Delta})Q(0,0) \\ + (\mu_{1,0}+\nu_{1,0}\sqrt{\Delta})Q_{1,0} + (\mu_{2,0}+\nu_{2,0}\sqrt{\Delta})Q_{2,0} + (\mu_{3,0}+\nu_{3,0}\sqrt{\Delta})Q_{3,0},
\end{multline}
where the $\mu$ and $\nu$ coefficients are polynomials in $t,x$ and the $\hat\mu$ and $\hat\nu$ coefficients are D-finite functions, being polynomials in $t,x$, and $\theta$ or $\theta_{0,0}$ respectively.

We next apply the canonical factorisation to $\Delta$. There are three roots, of which two are Puiseux series which converge at 0 and one is not:
\begin{align}
X_1 &= t + 2t^{5/2} + 6t^4 + 21t^{11/2} + 80t^7 + \textstyle \frac{1287}{4}t^{17/2} + 1344t^{10} + \bigO(t^{21/2}) \\
X_2 &= t - 2t^{5/2} + 6t^4 - 21t^{11/2} + 80t^7 - \textstyle \frac{1287}{4}t^{17/2} + 1344t^{10} + \bigO(t^{21/2}) \\
X_3 &= \textstyle \frac14 t^{-2} - 2t - 12t^4 - 160t^7 - 2688t^{10} + \bigO(t^{11})
\end{align}
Following \cref{lem:canonical_factorisation} we then define
\begin{align}
\Delta_0 &= 4t^2X_3 \\
\Delta_+ &= 1 - x/X_3 \\
\Delta_- &= (1-X_1\olx)(1-X_2\olx)
\end{align}
so that
\begin{align}
\frac{1}{\sqrt{\Delta_+}} &= 1 + 2t^2x + 6t^4x^2 + 16t^5x + \bigO(t^6) \\
\sqrt{\Delta_0\Delta_-} &= 1 - t\olx -4t^3 -2t^4\olx - 2t^5\olx^2 + \bigO(t^6).
\end{align}

As with reverse Kreweras walks, we next divide~\eqref{eqn:krew_Qx0_and_Qd1x} by $\sqrt{\Delta_+}$ and extract the $[x^>]$ and $[x^<]$ parts. Unfortunately this has the side effect of introducing $Q_{4,0}$ as another unknown. 
\begin{lemma}\label{lem:krew_two_kernel_equations}
\begin{align}
\sigma_{x,0}P_{x,0}Q(x,0) &= \hat\sigma + \hat\sigma_{0,0}Q(0,0) + \sigma_{1,0}Q_{1,0} + \sigma_{2,0}Q_{2,0} + \sigma_{3,0}Q_{3,0} + \sigma_{4,0}Q_{4,0} \label{eqn:krew_xpos} \\
\tau^\mathrm{d}_0P^\mathrm{d}_0Q^\mathrm{d}_0(x) &= \hat\tau + \hat\tau_{0,0}Q(0,0) + \tau_{1,0}Q_{1,0} + \tau_{2,0}Q_{2,0} + \tau_{3,0}Q_{3,0} + \tau_{4,0}Q_{4,0} \label{eqn:krew_xneg}
\end{align}
where 
\begin{align}
P_{x,0} &= \begin{multlined}[t] (ta - ta^2 - x + ax - t^2a^2x^2)(tb - tb^2 - x + bx - t^2b^2x^2) \\ \times (ta + tb - 2tab - 2x + ax + bx - 2t^2abx^2) \end{multlined} \\
P^\mathrm{d}_0 &= (t^2a^2 + x - ax - tax^2 + ta^2x^2) (t^2b^2 + x - bx - tbx^2 + tb^2x^2),
\end{align}
the $\sigma$ and $\tau$ coefficients are algebraic functions of $t,x$ (in fact $\sigma_{4,0}$ and $\tau_{4,0}$ are polynomials), and the $\hat\sigma$ and $\hat\tau$ coefficients are D-finite functions of $t$ with non-negative powers of $x$.
\end{lemma}

We then have analogue of \cref{lem:revkrew_twokernels_roots}.
\begin{lemma}\label{lem:krew_twokernels_roots}
The roots of $P_{x,0}$ are
\begin{align}
x_{1,2} &= \frac{\mp\sqrt{-4 a^4 t^3+4 a^3 t^3+a^2-2 a+1}+a-1}{2 a^2 t^2}\\
x_{3,4} &= \frac{\mp\sqrt{-4 b^4 t^3+4 b^3 t^3+b^2-2 b+1}+b-1}{2 b^2 t^2}\\
x_{5,6} &= \frac{\mp\sqrt{(a+b-2)^2-8 a b t^2 (2 a b t-a t-b t)}+a+b-2}{4 a b t^2}
\end{align}
and the roots of $P^\mathrm{d}_0$ are 
\begin{align}
x_{7,8} &= \frac{\mp\sqrt{-4 a^4 t^3+4 a^3 t^3+a^2-2 a+1}+a-1}{2at\left(a-1\right)}\\
x_{9,10} &= \frac{\mp\sqrt{-4 b^4 t^3+4 b^3 t^3+b^2-2 b+1}+b-1}{2bt\left(b-1\right)}.
\end{align}
If $a>1$ (resp.~$0<a<1$) then $x_1$ and $x_{7}$ (resp.~$x_2$ and $x_{8}$) are power series in $t$; if $b>1$ (resp.~$0<b<1$) then $x_3$ and $x_{9}$ (resp.~$x_4$ and $x_{10}$) are power series in $t$; and if $a+b>2$ (resp.~$0<a+b<2$) then $x_5$ (resp.~$x_{6}$) is a power series in $t$.
\end{lemma}

For any $a,b>0$ with $a\neq1,b\neq1$ and $a+b\neq2$, we thus get a set of five equations of the form
\begin{equation}\label{eqn:krew_after_cancelling_kernels}
0 = \zeta^{(i)} + \zeta^{(i)}_{0,0}Q(0,0) + \zeta^{(i)}_{1,0}Q_{1,0} + \zeta^{(i)}_{2,0}Q_{2,0} + \zeta^{(i)}_{3,0}Q_{3,0} + \zeta^{(i)}_{4,0}Q_{4,0} \equiv H_i,
\end{equation}
where $i$ indicates which $x_i$ root has been used in~\eqref{eqn:krew_xpos} or~\eqref{eqn:krew_xneg}.

Unfortunately, the determinant of the corresponding $5\times5$ matrix of coefficients appears to be 0. We can even include a sixth equation (say, $H_0$), by taking the $[x^0]$ part of~\eqref{eqn:krew_Qx0_and_Qd1x}, but there is still no set of five linearly independent equations.

\subsection{Incorporating the solution to reverse Kreweras walks}

By \cref{cor:krew00}, the generating function $Q(0,0)$ is the same for Kreweras and reverse Kreweras walks. Since we have solved reverse Kreweras walks, $Q(0,0)$ is actually known. We substitute it into~\eqref{eqn:krew_after_cancelling_kernels} and then we have six equations with four unknowns, and hence $15$ different equation sets.

\begin{proof}[Proof of \cref{thm:krew_dfinite}]
Substituting and expanding reveals that (at least) 10 of the equations sets yield a non-zero determinant. For example,
\begin{align}
D_{1,3,5,7} &= \frac{16a^{12}b^5 c^4 (a-1)^3(a-2)(b-1)^5  (a-b)^4 (a b-a-b) (2 a b-a-b)^5t^{26} }{(a+b-2)^4} + \bigO(t^{27})
\end{align}
(As with reverse Kreweras walks, this determinant depends on the the choice of $\sigma_{x,0}$ and $\tau^\mathrm{d}_0$.)

Any set of equations with a non-zero determinant then leads to a solution for $Q_{1,0}$, $Q_{2,0}$, $Q_{3,0}$ and $Q_{4,0}$. Back-substitution yields $Q(x,0)$, and by reflective symmetry, $Q(0,y)$. The original functional equation then gives $Q(x,y)$. If $a<1$ or $b<1$ then the roots which are not power series in $t$ can be replaced as required, and the resulting determinants remain non-zero.

Since $\zeta^{(i)}_{1,0},\zeta^{(i)}_{2,0},\zeta^{(i)}_{3,0}$ and $\zeta^{(i)}_{4,0}$ are algebraic while $\zeta^{(i)} + \zeta^{(i)}_{0,0}Q(0,0)$ is D-finite, the resulting solutions to $Q_{1,0},Q_{2,0},Q_{3,0}$ and $Q_{4,0}$ are D-finite. By back-substituting we find that the same is true for $Q(x,0), Q(0,y)$ and $Q(x,y)$. More specifically, the vector space over $\mathbb{C}(t,x,y,a,b,c)$ spanned by all partial derivatives of $Q(x,y)$ (with respect to $t,x,y$ only) is finite-dimensional. We do not (yet) have a proof that any of these functions are \emph{not} algebraic.
\end{proof}

\subsection{Special cases}

The special values of $a,b$ work in much the same way as for reverse Kreweras. When $a+b=2$ the third factor in $P_{x,0}$ loses its $x$ term, and as a result $x_{5,6}$ are not longer valid kernel roots. However, there remain three combinations which do not use these roots, and the solution can be obtained from any of those. 

If $b=1$ then $\sigma_{3,0} = \sigma_{4,0} = \tau_{3,0} = \tau_{4,0} = 0$, and we thus only need two independent equations to solve for the two unknowns $Q_{1,0}$ and $Q_{2,0}$. Using $H_1$ and $H_{7}$ (or $H_2$ and $H_{8}$) suffices. Naturally $a=1$ is just a reflection of the $b=1$ case.

When $a=b$, note that~\eqref{NCK} vanishes, and hence $\eta=\eta_{0,0}=\theta=\theta_{0,0}=0$. As was the case for reverse Kreweras walks, all coefficients in~\eqref{eqn:krew_Qx0_and_Qd1x} are then algebraic, and it follows that the resulting solution is too. Note that this has already been established for the $c=a^2=b^2$ case in~\cite{beaton_exact_2019}.

\section{Discussion}\label{sec:discussion}

\subsection{A generic equation}\label{6.1}

Observe that in solving both Kreweras and reverse Kreweras walks, an equation of the form
\begin{equation}\label{eqn:generalized_main_eqn}
\Lambda_F\Lambda_{FG} F(x)= \Lambda_{FG}\Lambda_G\sqrt{\Delta}G(\olx)+\sqrt{\Delta}A(x,U_1,\dots,U_k)+B(x,U_1,\dots,U_k)
\end{equation}
arose (namely, \eqref{eqn:rk_halforbitsum_y0} and \eqref{eqn:krew_Qx0_and_Qd1x}). Here,
\begin{itemize}
\item $F(x)$ and $G(\olx)$ are series in $t$ whose coefficients are polynomials in $x$ and $\olx$ respectively;
\item $\Lambda_F$, $\Lambda_{FG}$ and $\Lambda_G$ are products of distinct irreducible quadratic polynomials in $x$;
\item $\Delta$ is an irreducible cubic polynomial in $x$ (possibly divided by some power of $x$);
\item the $U_i$ are boundary terms of $F$ and $G$, ie.~$[x^j]F(x)$ or $[x^j]G(\olx)$ for some $j$; and
\item $A(\cdots)$ and $B(\cdots)$ are linear combinations of unknowns $U_1,\dots,U_k$ with coefficients that are polynomial in $x,t$ in~\eqref{eqn:rk_halforbitsum_y0} and D-finite in~\eqref{eqn:krew_Qx0_and_Qd1x}.
\end{itemize}

For both models we then obtained two kernel-like equations by taking either the positive or negative part of~\eqref{eqn:generalized_main_eqn} with respect to $x$. For reverse Kreweras this led to exactly three linearly independent equations in the $U_i$, and since $k=3$, this gave the solution. For Kreweras we obtained four independent equations, but while we initially had $k=4$, a fifth boundary term emerged when taking the positive and negative parts. It was only because we already knew the (algebraic) solution to one of the $U_i$ that we were able to solve the system.

A generic equation of this form can be contrasted with the results of~\cite{bousquet2006polynomial}. Suppose we have a single polynomial equation
\begin{equation}
P(Q(x,t),Q_1,\dots,Q_k,t,x)=0,\label{eq}
\end{equation}
and seek to determine all unknowns $Q(x,t)$ and $Q_k$. It has been proved that under some mild assumptions regarding the form of the equation, if one can find $k$ distinct formal series $X_i(t)$ ($i=1,\dots,k)$ that make
\begin{equation}
\left.\frac{\partial P}{\partial x_0}P(x_0,Q_1,\dots,Q_k,t,X_i)\right|_{x_0=Q(X_i,t)} =0,
\end{equation}
then these $k+1$ series are algebraic. (For a kernel-like equation obtained from a lattice walk problem, Banderier and Flajolet~\cite{banderier2002basic} point out that the determinant arising from the kernel method forms a 
Vandermonde determinant.)

The conditions under which~\eqref{eqn:generalized_main_eqn} is solvable seem to be more complicated. The positive and negative parts can be written as 
\begin{equation}\label{eqn:generalized_pos_eqn}
\Lambda_F\Lambda_{FG} \frac{F(x)}{\sqrt{\Delta_+}} =  G_+(x,U_1,\dots,U_k,\dots,U_\ell) + A_+(x,U_1,\dots,U_k) + \frac{B(x,U_1,\dots,U_k)}{\sqrt{\Delta_+}}
\end{equation}
\begin{multline}\label{eqn:generalized_neg_eqn}
0 = \Lambda_{FG}\Lambda_G\sqrt{\Delta_0\Delta_-}G(\olx) - G_+(x,U_1,\dots,U_k,\dots,U_\ell) \\ + \sqrt{\Delta_0\Delta_-}A(x,U_1,\dots,U_k)-A_+(x,U_1,\dots,U_k)
\end{multline}
where
\begin{itemize}
\item $\Delta = \Delta_0\Delta_+\Delta_-$ is the canonical factorisation (see \cref{lem:canonical_factorisation});
\item $G_+(\cdots)$ is the non-negative part of $\Lambda_{FG}\Lambda_G\sqrt{\Delta_0\Delta_-}G(\olx)$ -- this can be expressed as a linear combination of boundary terms of $G(\olx)$  with polynomial coefficients; and
\item $A_+(\cdots)$ is the non-negative part of $\sqrt{\Delta_0\Delta_-}A(x)$ -- this is a linear combination of the unknowns $U_i$ with coefficients of the same type as in $A(\cdots)$ (ie.~polynomials or D-finite functions).
\end{itemize}

Generically, this system is solvable if exactly $\ell$ linearly independent equations can be obtained by substituting values of $x$ which cancel $\Lambda_F$, $\Lambda_{FG}$ or $\Lambda_G$, while leaving $F(x)$ or $G(\olx$) as power series in $t$. We have yet to understand exactly what constitutes a set of sufficient conditions to guarantee this.

\subsection{Comparison with other models}

We may also ask for which other models an equation of the form~\eqref{eqn:generalized_main_eqn} arises. In particular, since the separation of~\eqref{eqn:generalized_main_eqn} into positive and negative parts depends on the canonical factorisation of $\sqrt{\Delta}$, the coefficient of $G(\olx)$ has to be a polynomial or a polynomial multiplying $\sqrt{\Delta}$. A factor like $C(x,t)+D(x,t)\sqrt{\Delta}$ (with $C,D\in \mathbb{R}(x,t)$) will prevent us from finding a solution. This is because in an equation of the form
\begin{equation}
\frac{\Lambda_F\Lambda_{FG}}{\Delta_+} F(x)=\left(\frac{C}{\sqrt\Delta_+}+D\sqrt{\Delta_0\Delta_-}\right)G(\olx)+\sqrt{\Delta_0\Delta_-}A(x,U_1,\dots,U_k)+\frac{B(x,U_1,\dots,U_k)}{\sqrt\Delta_+},
\end{equation}
the first term in the RHS is still a product of a formal series in $x$ and a series in $\olx$.

For those models in~\cite{beaton_exact_2019} which were not solved exactly (except Kreweras and reverse Kreweras, which we have solved here), it is straightforward to show that the determinant of $\mathbf M$ (recall \eqref{K(x,y)Q=MV+C}) is nonzero. That is, it is not possible to eliminate all terms of the form $Q(\bullet,0)$ and $Q(0,\bullet)$ in a full-orbit sum. Without loss of generality, assume that $Q(x,0)$ is the term left on the RHS of the full-orbit sum. The full-orbit sum can be written as
\begin{equation}
\mathbf{N}\mathbf{Q}=\frac{P_1(Q(x,0),x,y,t)}{K(x,y)},
\end{equation}
where $P_1$ is a linear function of $Q(x,0)$ with polynomial coefficients in $x,y,t$. When taking the $[y^0]$ part of this equation, the LHS is similar to \eqref{rev_halfOS_lhs_y0}:
\begin{equation}\label{sec6 orbit-sum}
[y^0]\mathbf{NQ} = \delta + \delta_{0,0}Q(0,0) + \delta_{x,0}Q(x,0)  + \delta^\mathrm{d}_0 Q^\mathrm{d}_0(\olx).
\end{equation}
We have $Q^\mathrm{d}_0(\olx)$ and $Q(x,0)$.
All the coefficients on the LHS are polynomials and do not contain $\sqrt{\Delta}$. But on the RHS, we must apply \cref{l1} when finding the $[y^0]$ part, and therefore we have
\begin{equation}
[y^0]\frac{P_1(Q(x,0),x,y,t)}{K(x,y)}=\frac{1}{\sqrt{\Delta}}P_2(Q(x,0),x,t).
\end{equation}
The coefficient of $Q(x,0)$ contains $\sqrt{\Delta}$. Combining the LHS and RHS together, we will get an equation with polynomial coefficients on $Q^\mathrm{d}_0(\olx)$ and $C+D\sqrt{\Delta}$ on $Q(x,0)$, or vice versa.

We now turn to the cases which have been solved using the kernel method, both in this paper and in~\cite{beaton_exact_2019}. For all algebraic cases that we know currently~\cite{beaton_exact_2019,bousquet2010walks}, the full-orbit sum is $0$. The half-orbit sum leads to an equation in the form of~\eqref{eqn:generalized_main_eqn} (except Gessel walks, which are solved by different ideas in \cite{bousquet2016elementary} and \cite{kurkova2011explicit}) where the coefficient of $Q^\mathrm{d}_0(\olx)$ is $\sqrt{\Delta}$ times a polynomial. We can walk through the algebraic kernel method and solve the problem.

For all D-finite (except Kreweras walks) and D-algebraic cases, the RHS of the full-orbit sum does not vanish. But it does not contain $Q(x,0)$. The orbit sum can be written as
\begin{equation}
\mathbf{N}\mathbf{Q}=\frac{P_1(x,y)}{K(x,y)}.
\end{equation}
After taking the $[y^0]$ part of it, we have
\begin{equation}
[y^0]\frac{P_1(x,y)}{K(x,y)}=\frac{1}{\sqrt{\Delta}}P_2(x,t).
\end{equation}
The LHS can still be represented by \eqref{sec6 orbit-sum}. We do have $\sqrt\Delta$ in the RHS but the coefficients of $Q(x,0)$ and $Q^d_0(\olx)$ all come from the LHS. So the coefficient of $Q^d_0(\olx)$ is a polynomial. We cannot divide the $\sqrt{\Delta_+}$ on both hand side since the LHS contains $Q^d_0(\olx)$. However we can directly take the $[x^>]$ or $[x^<]$ part of this equation. This leads to the D-finite terms on the RHS.

After the separation one attempts to cancel the kernels and (hopefully) solve a set of linear equations. Let us temporarily ignore the previous discussion in \cref{6.1}, and assume that we can find a solution. This means if $A,B$ and the $\Lambda$ terms are algebraic in~\eqref{eqn:generalized_main_eqn}, the final solution is algebraic. If they are D-finite, the final solution is D-algebraic. In order to distinguish the D-finite and D-algebraic cases, we need check the coefficients of the unknowns $U_{i}$ in this equation. By the process of the algebraic kernel method, we find all coefficients of $U_i$ are polynomials of $x,t$ except possibly $Q(0,0)$. This is because $Q(0,0)$ is the only one of the $U_i$ which can appear in both sides of the full-orbit sum. 
So, we conclude, if the coefficient of $Q(0,0)$ on the RHS of the full-orbit sum is $0$, the solution is D-finite. Otherwise, the solution is D-algebraic.

This implies that the values of $a,b,c$ can affect the nature of the generating function. For example, consider simple random walks with step set $\{\mathrm{N, E, S, W}\}$. The functional equation reads
\begin{multline}
K(x,y)Q(x,y)=\frac1c + \frac1a(a-1 - ta\oly)Q(x,0) + \frac1b(b - 1 - tb\olx)Q(0,y) \\
+\frac{1}{abc}(ac+bc-ab-abc)Q(0,0)
\end{multline}
Using the algebraic kernel method, we find the RHS of the full-orbit sum to be
\begin{equation}
-\frac{\left(x^2-1\right) \left(y^2-1\right) z^2 [ (a b - a c - b c + a b c)Q(0,0)-a b]}{c x yK(x,y)}.
\end{equation}
Then, taking the $[y^0]$ part of both sides, the coefficient of $Q(0,0)$ in the RHS is 
\begin{equation}
z^2(a b - a c - b c + a b c)[y^0]\left(\frac{(x^2-1)(y^2-1)}{cxyK(x,y)}\right).
\end{equation}
Notice that $[y^0]\left(\frac{(x^2-1)(y^2-1)}{cxyK(x,y)}\right)$ is D-finite (since we take the positive part in $x$ of a rational function). If $ab-ac-bc+abc=0$, the solution is D-finite. Otherwise, as discussed above, it is D-algebraic. This is an improvement on the result obtained with the obstinate kernel method in \cite[Theorem 1]{beaton_exact_2019}.

We have just discussed the general properties of walks with and without boundary interactions. These simple inferences hold for most models we have studied, except Kreweras walks. Kreweras walks appear to be rather special. Kreweras walks without boundary interactions are algebraic. When interactions are introduced, the full-orbit sum (see \eqref{krew_fullOS_lhs_xposy0}) behaves more like a D-algebraic model, as the RHS is not $0$. Unlike the D-algebraic cases, however, since $Q(x,0)$ and $Q(0,x)$ both appear in the LHS of full-orbit sum, we still need the half-orbit sum to cancel $Q(0,x)$. The half-orbit sum then behaves more like an algebraic model. For example, observe that~\eqref{eqn:krew_Qx0_and_Qd1x} is in the same form as~\eqref{eqn:rk_halforbitsum_y0}, and the coefficient of $Q^\mathrm{d}_0(\olx)$ is a polynomial times $\sqrt{\Delta}$. However, since some terms in \eqref{eqn:krew_Qx0_and_Qd1x} are D-finite and not algebraic, the final result ends up being D-finite.

However, in another sense, Kreweras walks are different to the other D-algebraic models. The coefficient of $Q(0,0)$ in the full-orbit sum~\eqref{NCK} does not affect the properties of the solution -- it will always be D-finite.

We summarise this discussion in \cref{tab 1to4}. We know, for non-interacting models, the algebraic and D-algebraic cases can be distinguished just by examining the full-orbit sum. However, for interacting models, this rule is broken by Kreweras walks. Is it still the case that when the full-orbit sum equals $0$, the solution must be algebraic? And is there an easy way to determine when the solution is D-finite or D-algebraic?

It is of great interest to know whether other recently discovered techniques which do not use the kernel method -- for example, the elliptic functions of Kurkova and Raschel~\cite{kurkova_new_2015} -- may be applied to the problems of quarter-plane lattice paths with interacting boundaries.

\newcolumntype{L}[1]{>{\raggedright\let\newline\\\arraybackslash\hspace{0pt}}m{#1}}

\begin{table}[ht]
\small
\centering
\begin{tabular}{|L{2.5cm}|L{2.9cm}|L{2.9cm}|L{2.9cm}|L{2.9cm}|}
  \hline
 {Model} & Full-orbit sum & Coefficient of $Q_0^d(\olx)$ in $[y^0]$ of full-orbit sum & Coefficient of $Q_0^d(\olx)$ in $[y^0]$ of half-orbit sum & Generating function $Q(x,y,t)$\\
  \hline
  D-finite model without interactions & Linear equation of $Q(0,0)$ with coefficients in $\mathbb{R}(x,y,t)$ & Polynomial & Not needed & D-finite \\
  \hline
Algebraic model without interactions & 0 & 0 & $D\sqrt\Delta$ & Algebraic \\
  \hline
D-algebraic model with interactions & Linear equation of $Q(0,0)$ with coefficients in $\mathbb{R}(x,y,t)$ & Polynomial & Not needed & D-algebraic (may become D-finite for certain $a,b,c$) \\
   \hline
Reverse Kreweras walks with interactions & 0 & 0 & $D\sqrt\Delta$ & Algebraic \\
  \hline
Unsolved models with interactions & Linear equation of $Q(x,0),Q(0,0)$ with coefficients in $\mathbb{R}(x,y,t)$ & $C+D\sqrt\Delta$ & Not needed & Unclear  \\
  \hline
Kreweras walks with interactions & Linear equation of $Q(0,0)$ with coefficients in $\mathbb{R}(x,y,t)$ & Polynomial & $D\sqrt\Delta$ & D-finite for all $a,b,c$ \\
  \hline
\end{tabular}
\caption{A summary of the various model types. Here $C$ and $D$ are in $\mathbb{R}(x,t)$.}
\label{tab 1to4}
\end{table}

\subsection{Singularities and asymptotic behaviour}

We have so far only been concerned with the solutions to the generating functions for Kreweras and reverse Kreweras walks, without considering any physical or probabilistic meaning for the weights $a,b,c$. However, as mentioned in \cref{sec:intro}, there are close connections between lattice paths and models in statistical physics and probability theory. It is thus also worth investigating the asymptotic behaviour of the coefficients of the generating functions. This is ongoing work.

\section{Conclusion}\label{conclusion}

We have introduced a model of quarter-plane lattice paths with weights $a,b,c$ associated with visits to the two boundaries and the origin. We have then used the algebraic kernel method to solve the particular cases of reverse Kreweras walks and Kreweras walks. 

The final solution of reverse Kreweras walks is algebraic and we have directly solved it for all $(a,b,c)$ (without explicitly writing out the expression, due to its complexity). However, we were unable to solve Kreweras walks directly, and instead needed to make use of the fact that when only considering walks which start and end at the origin, Kreweras and reverse Kreweras generating functions are identical. The overall solution for Kreweras walks is D-finite.

There remain many other quarter-plane lattice path models which have not been solved, and it is unclear which methods may be useful when boundary weights are included. How the values of $a,b,c$ affect the asymptotic behaviour of these models is also an open question.

\section*{Acknowledgements}
NRB and ALO gratefully acknowledge support from the Australian Research Council.

\printbibliography

\end{document}